\documentclass[12pt]{amsart}
\usepackage{amsmath}
\usepackage[bmargin=1.3in]{geometry}
\usepackage{amsthm}
\usepackage{amsfonts}
\usepackage{comment}
\usepackage{amssymb}
\usepackage{eufrak}
\usepackage{comment}
\usepackage{hyperref}
\usepackage{mathtools}
\usepackage{subdepth}

\numberwithin{equation}{section}
\usepackage{breqn}

\newtheorem{Theorem}{Theorem}[section]

\newtheorem{Lemma}[Theorem]{Lemma}

\theoremstyle{definition}
\newtheorem*{Example}{Example}
\theoremstyle{remark}
\newtheorem*{remark}{Remark}

\usepackage{cite}

\title[proof of Rose's conjecture]{Proof of conjecture regarding the level of  Rose's generalized sum-of-divisor functions}

\author[Hannah Larson]{Hannah Larson}
\address{Department of Mathematics, Harvard University, Cambridge, MA 02138}
\email{hannahlarson@college.harvard.edu}

\begin{document}

\maketitle

\begin{abstract}
In a recent paper, Rose proves that certain generalized sum-of-divisor functions are quasi-modular forms for some congruence subgroup and conjectures that these forms are quasi-modular for $\Gamma_1(n)$. Here, we prove this conjecture.
\end{abstract}

\section{Introduction}

In \cite{R}, Rose studies generalizations of MacMahon's sum-of-divisor functions and shows that they are quasi-modular forms. For the context and role of these functions, we refer the reader to this paper and references cited therein.
Rose's generalized functions are defined with respect to symmetric sets $S \subset \{1, \ldots, n\}$ of residues modulo $n$ satisfying the property that for all $\ell \in S$, we also have $- \ell \in S \pmod n$. For such a set $S$ and any positive integer $k$, define
\begin{align*}
A_{S,n,k}(q) := \sum_{\substack{0 < m_1 < \cdots < m_k \\ m_i \in S \pmod n}} \frac{q^{m_1+\ldots+m_k}}{(1-q^{m_1})^2 \cdots (1 - q^{m_k})^2}.
\end{align*}
Rose proves that $A_{S,n,k}(q)$ is a quasi-modular form of mixed weight for some congruence subgroup $\Gamma \subseteq \mathrm{SL}_2(\mathbb{Z})$ and conjectures that $A_{S,n,k}(q)$ is actually quasi-modular for $\Gamma_1(n)$. We prove that this is indeed the case. More precisely, our result is the following.

\begin{Theorem} \label{con}
The functions $A_{S,n,k}(q)$ are quasi-modular forms of mixed weight at most $2k$ for $\Gamma_1(n)$.
\end{Theorem}
 
 We present an example of the forms that arise when $n=5$ and $k=1$.
 
 \begin{Example}
By Theorem \ref{con}, the forms $A_{S,5,1}(q)$ are quasi-modular forms of mixed weight at most $2$ on $\Gamma_1(5)$. If we let 
$\chi$ be the principal Dirichlet character modulo $5$ and $\psi= \left(\frac{\bullet}{5}\right)$, this space is spanned by constants and the Eisenstein series
\begin{align*}
G_2(q) &= -\frac{B_2}{4} + \sum_{m=1}^\infty \left(\sum_{n \mid m} n\right)q^m = -\frac{1}{24}+q+3q^2+4q^3+7q^4+6q^5+\ldots \\
G_2(q^5) &= -\frac{1}{24}+q^5+3q^{10}+4q^{15}+7q^{20}+\ldots 
\end{align*}
\begin{align*}
G_{2,\chi,\psi}(q) &=-\frac{B_{2,\psi}}{4}+ \sum_{m=1}^\infty \left(\sum_{n \mid m} \psi(n) \chi\left(\tfrac{m}{n}\right) n\right)q^m=
-\frac{1}{5} +q-q^2-2q^3+3q^4+q^5+\ldots \\
G_{2,\psi,\chi}(q) &= \sum_{m=1}^\infty\left(\sum_{n\mid m} \chi(n)\psi\left(\tfrac{m}{n}\right)n\right)q^m=
q+q^2+2q^3+3q^4+5q^5+\ldots.
\end{align*}
The following table lists an expression for $A_{S,5,1}(q)$ in terms of these Eisenstein series for each symmetric set $S$ modulo $5$.

\vspace{.05in}
\begin{center}
\begin{tabular}{|l|l|}
\hline
\rule{0pt}{2.5ex}$S$ & $A_{S,5,1}(q)$ \\[0.5ex] 
\hline
\hline
\rule{0pt}{2.5ex}$\{5\}$ & $G_2(q^5)+\frac{1}{24}$ \\ [0.5ex]
\hline
\rule{0pt}{2.5ex}$\{1,4\}$ & $\frac{1}{2}G_2(q) - \frac{1}{2}G_2(q^5) + \frac{1}{2}G_{2,\psi,\chi}(q)$ \\ [0.5ex]
\hline
\rule{0pt}{2.5ex}$\{1,4,5\}$ & $\frac{1}{2}G_2(q) + \frac{1}{2}G_2(q^5) + \frac{1}{2}G_{2,\psi,\chi}(q)+\frac{1}{24}$ \\[0.5ex]
\hline
\rule{0pt}{2.5ex}$\{2,3\}$ & $\frac{1}{2}G_2(q) -\frac{1}{2}G_2(q^5) -\frac{1}{2}G_{2,\psi,\chi}(q)$\\ [0.5ex]
\hline
\rule{0pt}{2.5ex}$\{2,3,5\}$ & $\frac{1}{2}G_2(q) +\frac{1}{2}G_2(q^5) -\frac{1}{2}G_{2,\psi,\chi}(q) + \frac{1}{24}$\\[0.5ex]
\hline
\rule{0pt}{2.5ex}$\{1,2,3,4\}$ & $G_2(q) - G_2(q^5)$ \\ [0.5ex]
\hline
\rule{0pt}{2.5ex}$\{1,2,3,4,5\}$ & $G_2(q) + \frac{1}{24}$ \\ [0.5ex]
\hline
\end{tabular}
\end{center}
 \end{Example}
 
 \begin{remark}
 The constant summands above are permitted as the weight-$0$ components of $A_{S, n, k}(q)$.
 \end{remark}
 
Rose proves that the $A_{S,n,k}(q)$ are quasi-modular by relating them to the Taylor coefficients of a particular Jacobi form. Since the Taylor coefficients of a Jacobi form are quasi-modular for the same group, we study the modularity properties of this Jacobi form to deduce the desired properties of $A_{S,n,k}(q)$. In Section \ref{jf}, we review the theory of Jacobi forms as in \cite{EZ} and recall Rose's results in terms of these functions. Then in Section \ref{pf}, we study the modularity properties of the Jacobi forms and nearly-holomorphic modular forms that arise in Rose's work to prove the conjecture.

\subsection*{Acknowledgements}
This research was carried out during the 2015 REU at Emory University. The author would like to thank Ken Ono for pointing out Rose's recent work, Michael Mertens, Sarah Trebat-Leader, Eric Larson, and Michael Griffin  for useful conversations, and the NSF for its support. The author also thanks an anonymous referee for helpful comments on an earlier version of this paper.

\section{Jacobi forms and Rose's work} \label{jf}

Given a matrix $\gamma = \left(\begin{smallmatrix} a&b\\c&d \end{smallmatrix}\right) \in \mathrm{SL}_2(\mathbb{Z})$, the weight-$k$ \textit{slash operator} acts on functions on the upper half-plane  $\mathbb{H}$ by
\begin{equation}
(f\vert_k\gamma)(\tau)  := (c\tau+d)^{-k}f\left(\frac{a\tau+b}{c\tau+d}\right).
\end{equation}
A \textit{nearly-holomorphic modular form} of weight $k$ on a subgroup $\Gamma \subseteq \mathrm{SL}_2(\mathbb{Z})$ is a function $f:\mathbb{H}\rightarrow \mathbb{C}$ which satisfies $(f\vert_k\gamma)(\tau) = f(\tau)$ for all $\gamma \in \Gamma$ and is
expressible as a polynomial in $\frac{1}{y}$, where $y := \mathrm{Im}(\tau)$,
with holomorphic coefficients. A \textit{quasi-modular form} is the constant term, with respect to $\frac{1}{y}$, of a nearly-holomorphic modular form.

The \textit{Jacobi group} is the set of triples 
\[G^J := \left\{[M,X,\xi] \in \mathrm{SL}_2(\mathbb{R}) \times \mathbb{R}^2 \times \mathbb{C} : |\xi| = 1\right\},\] under the group law
\begin{equation}\label{gplaw}
\left[\begin{array}{c}\\ \end{array} \! \! \! \! \! M,X,\xi\right]\left[M',X',\xi'\right] = \left[MM', XM'+X',\xi\xi' \cdot e\left(\det\left(\begin{matrix} XM' \\ X' \end{matrix}\right) \right)\right],
\end{equation}
where $e(x) := e^{2\pi i x}$.
Similar to the slash operator, given integers $k$ and $m$, the Jacobi group acts on the space of holomorphic functions $\phi: \mathbb{H}\times \mathbb{C} \rightarrow \mathbb{C}$ by
\begin{align}
&\left({\!\!\!\! \!\left.\begin{array}{c} \\[1.94ex] \end{array} \phi\right|}_{k,m} \left[ \left(\begin{matrix} a & b \\ c& d\end{matrix}\right), (\lambda, \mu), \xi\right]\right)(\tau,z) \notag \\
&\qquad \qquad :=\xi^m(c\tau+d)^{-k} e\left(\frac{-cm(z+\lambda\tau+\mu)^2}{c\tau+d}+m(\lambda^2\tau+2\lambda z+\lambda\mu) \right)\\
&\qquad\qquad \qquad \times \phi\left(\frac{a\tau+b}{c\tau+d},\frac{z+\lambda\tau+\mu}{c\tau+d}\right). \notag
\end{align}
 (See Theorem 1.4 of \cite{EZ}). We often drop the subscripts $k$ and $m$ when they are clear from context.
%\begin{remark}
%This definition can be extend to half-integers $k$ and $m$ by inserting a factor of $i^k$ for the transformation by $\left(\begin{smallmatrix} 0 & -1 \\ 1 & 0 \end{smallmatrix}\right)$ and modifying transformations by other group elements accordingly.
%\end{remark}
To simplify notation, we write $I = \left(\begin{smallmatrix} 1 & 0 \\ 0 & 1 \end{smallmatrix}\right)$ for the identity matrix, $[M]$ for $[M, (0,0),1],$ and $[X]$ for $[I, X,1]$.
A \textit{Jacobi form} of weight $k$ and index $m$ for a subgroup $\Gamma \subseteq \mathrm{SL}_2(\mathbb{Z})$ is a holomorphic function $\phi: \mathbb{H} \times \mathbb{C} \rightarrow \mathbb{C}$ which satisfies
\[(\phi|_{k,m}[M])(\tau,z) = \phi(\tau,z)\]
for all $M \in \Gamma$ and
\[(\phi|_{k,m}[X])(\tau,z) = \phi(\tau, z)\]
for all $X \in \mathbb{Z}^2$.

\begin{remark}
One also requires certain growth conditions at the cusps. We refer the reader to \cite{EZ} for more details on Jacobi forms.
\end{remark}

The \textit{Jacobi theta function}, defined by
\begin{equation}
\vartheta(\tau,z) := \sum_{n \in \mathbb{Z}}e\left(\tfrac{1}{2}n^2\tau + n z\right) = \sum_{n \in \mathbb{Z}}q^{\frac{1}{2}n^2}\zeta^n,
\end{equation}
where $q:=e^{2\pi i\tau}$ and $\zeta := e^{2\pi i z}$, is a Jacobi form of weight $\frac{1}{2}$ and index $\frac{1}{2}$
for the subgroup $\left\langle \left(\begin{smallmatrix} 1 & 2 \\ 0 & 1 \end{smallmatrix}\right), \left(\begin{smallmatrix} 0 & -1 \\ 1 & 0 \end{smallmatrix}\right) \right\rangle\subset \mathrm{SL}_2(\mathbb{Z}).$ 
%From this definition, it is easy to see that
%\[\vartheta|[\left(\begin{smallmatrix} 1&2\\0&1 \end{smallmatrix}\right)](\tau,z) = \vartheta(\tau,z) \qquad \text{and} \qquad \vartheta|[X](\tau,z) = \vartheta(\tau,z)\]
%for all $X \in \mathbb{Z}^2$.
%In addition, it is well-known that
%\[\vartheta |\left[\left(\begin{smallmatrix} 0 & -1 \\ 1 & 0 \end{smallmatrix}\right) \right](\tau, z) = \vartheta(\tau,z).\]
We note that
\begin{equation*}
\left(\vartheta|\left[\left(\begin{smallmatrix} 1 & 1 \\ 0 & 1 \end{smallmatrix}\right)\right]\right)(\tau, z) = \sum_{n \in \mathbb{Z}} (-1)^nq^{\frac{1}{2}n^2}\zeta^n = \left(\vartheta|\left[\left(0, \tfrac{1}{2}\right)\right]\right)(\tau, z).
\end{equation*}
In particular, the above implies that the function defined by
\begin{equation}
\psi(\tau,z) := \frac{\vartheta(\tau,z)}{\vartheta(\tau,0)}
\end{equation}
is a Jacobi form of weight $0$ and index $\frac{1}{2}$ on the subgroup 
$\langle \left(\begin{smallmatrix} 1 & 2 \\ 0 & 1  \end{smallmatrix}\right), \left(\begin{smallmatrix}0 & -1 \\ 1 & 0 \end{smallmatrix}\right)\rangle \subset \mathrm{SL}_2(\mathbb{Z})$
which satisfies
\begin{equation} \label{tr}
\left(\psi |\left[\left(\begin{smallmatrix} 1 & 1 \\ 0 & 1 \end{smallmatrix}\right)\right]\right)(\tau, z) = \left(\psi|\left[\left(0, \tfrac{1}{2}\right)\right]\right)(\tau, z).
\end{equation}

As in \cite{R}, let
\begin{equation}
\vartheta_r(\tau, z) := \left(\vartheta|\left[\left(r, 0\right)\right] \right)(\tau,z)=\sum_{m \in \mathbb{Z}+r} q^{\frac{1}{2}m^2}\zeta^m.
\end{equation}
Similarly, define
\begin{equation}
\psi_r(\tau,z) := \left(\psi|\left[\left(r, \tfrac{1}{2}\right)\right] \right)(\tau,z) = \frac{\vartheta_r(\tau, z+\tfrac{1}{2})}{\vartheta_r(\tau, \tfrac{1}{2})} = \frac{\vartheta_r(q, -\zeta)}{\vartheta_r(q,-1)},
\end{equation}
where the last term is written to match Rose's notation.
In addition, as in \cite{R}, let
\[\alpha(n, \ell) := \frac{\ell}{n} - \frac{1}{2},\]
let $D$ be the differential operator $\frac{1}{2\pi i} \frac{d}{d\tau} = q \frac{d}{dq}$, and let
\[\eta(\tau) = q^{\frac{1}{24}}\prod_{m=1}^\infty (1 - q^m)\]
be the \textit{Dedekind eta-function}.
Given a symmetric set $S \subset \{1, \ldots, n\}$, the pure-weight parts of $A_{S,n,k}(q)$ are expressible in terms of  the Jacobi forms
\[\phi_S(\tau, z) := \prod_{\ell \in S\backslash n}\psi_{\alpha(n, \ell)}(\tau, z).\]
With this notation, we can restate the relevant parts of Rose's results as follows (see proofs of Theorems 1.11 and 1.12 in \cite{R}).
\begin{Theorem}[Rose]
Let $S \subset \{1, \ldots, n\}$ be a symmetric set of residues modulo $n$. The function $A_{S,n,k}(q)$ is a quasi-modular form of mixed weight at most $2k$.
In particular, when $n \notin S$, the weight $2w$ part of $A_{S,n,k}(q)$ is a multiple of
\[\left(\frac{\partial}{\partial z}\right)^{2w}\!\!\!\!\!\!\left.\begin{array}{c} \\[2ex] \end{array} \phi_S(n\tau,z)\right|_{z=0},\]
and when $n \in S$, the weight $2w$ part is a multiple of
\[\sum_{i=1}^w{2w+1 \choose 2i+1} \left(\frac{2}{n}\right)^{i} \frac{D^{i}\eta(n\tau)^3}{\eta(n\tau)^3} \cdot \left(\frac{\partial}{\partial z}\right)^{2w-2i}\!\!\!\!\!\!\left.\begin{array}{c} \\[2ex] \end{array} \phi_S(n\tau,z)\right|_{z=0}.\]
\end{Theorem}

The key relationship between Jacobi forms and quasi-modular forms is that the Taylor coefficients for the expansion around $z=0$ of a Jacobi form are quasi-modular forms for the same group (see equation (6) on p.~31 of \cite{EZ}). More precisely, we have the following.

\begin{Lemma}[Corollary 2.3 of \cite{R}] \label{2.2}
Given any Jacobi form $\phi(\tau, z)$ of weight $k$ for some group $\Gamma \subseteq \mathrm{SL}_2(\mathbb{Z})$ and any positive integer $m$, the function
\[\left.\left(\frac{\partial}{\partial z}\right)^m\phi(\tau,z) \right|_{z=0}\]
is a quasi-modular form of weight $k+m$ for the same group $\Gamma$. That is, there is a nearly-holomorphic modular form $\Phi^{(m)}(\tau)$ of weight $k+m$ on $\Gamma$ whose constant term with respect to $\frac{1}{y}$ is the function above.
\end{Lemma}

For a symmetric set $S$ and positive integer $w$, we define $\Phi_S^{(2w)}(\tau)$ to be the nearly-holomorphic modular form whose constant term with respect to $\frac{1}{y}$ is
\[\left(\frac{\partial}{\partial z}\right)^{2w}\!\!\!\!\!\!\left.\begin{array}{c} \\[2ex] \end{array} \phi_S(\tau,z)\right|_{z=0}.\]
To prove Rose's conjecture, we need to show that $\Phi_S^{(2w)}(n\tau)$ is modular for $\Gamma_1(n)$. 

\section{Proof of conjecture} \label{pf}

By Lemma \ref{2.2}, the modularity properties of $\Phi_S^{(2w)}(\tau)$ are determined by those of the Jacobi form $\phi_S(\tau, z)$.
We first determine congruence conditions on matrices which are sufficient for them to fix $\phi_S(\tau,z)$, and hence $\Phi_S^{(2w)}(\tau)$.
Recall that for a positive integer $n$, the \textit{principal congruence subgroup} of level $n$ is defined by
\[\Gamma(n) := \left\{\left(\begin{smallmatrix} a & b \\ c & d \end{smallmatrix}\right) \in \mathrm{SL}_2(\mathbb{Z}) : \left(\begin{smallmatrix} a & b \\ c & d \end{smallmatrix}\right) \equiv \left(\begin{smallmatrix} 1 & 0 \\ 0 & 1 \end{smallmatrix}\right)\!\!\!\! \pmod n \right\},\]
and the congruence subgroup $\Gamma_1(n)$ is defined by
\[\Gamma_1(n):=\left\{\left(\begin{smallmatrix} a & b \\ c & d \end{smallmatrix}\right) \in \mathrm{SL}_2(\mathbb{Z}) : \left(\begin{smallmatrix} a & b \\ c & d \end{smallmatrix}\right) \equiv \left(\begin{smallmatrix} 1 & * \\ 0 & 1 \end{smallmatrix}\right)\!\!\!\! \pmod n \right\}.\]
Throughout this section, fix a symmetric set $S \subset \{1, \ldots, n\}$ of residues modulo $n$.

\begin{Lemma} \label{3.1}
Let $\left(\begin{smallmatrix}a & b \\ c&d \end{smallmatrix}\right) \in \Gamma(2)$ with $a \equiv 1 \pmod n$ and $b \equiv 0 \pmod n.$
Then
\[\left(\phi_S|\left[\left(\begin{smallmatrix} a & b \\ c & d \end{smallmatrix}\right)\right]\right)(\tau,z) = \phi_S(\tau, z).\]
In particular, for any positive integer $w$, we have
\[\left(\left.\Phi_S^{(2w)}\right|\left(\begin{smallmatrix} a & b \\ c & d \end{smallmatrix}\right)\right)(\tau) =\Phi_S^{(2w)}(\tau).\]
\end{Lemma}

\begin{proof}
First note that
% $\Gamma(2)$ is generated by
%\begin{align*}
%\left(\begin{matrix} 1 & 2 \\ 0 & 1 \end{matrix} \right), \qquad\left(\begin{matrix} -1 & 0 \\ 0 & -1 \end{matrix}\right) = \left(\begin{matrix} 0 & -1 \\ 1 & 0 \end{matrix}\right)\left(\begin{matrix} 0 & -1 \\ 1 & 0 \end{matrix}\right),
%\end{align*}
%and
%\[\left(\begin{matrix} -1& 0 \\ 2 & -1 \end{matrix}\right)
%= \left(\begin{matrix} 0 & -1 \\ 1 & 0 \end{matrix}\right)
%\left(\begin{matrix} 1 & 2 \\ 0 & 1 \end{matrix} \right)
%\left(\begin{matrix} 0 & -1 \\ 1 & 0 \end{matrix}\right),
%\]
%so 
$\Gamma(2) \subset \left\langle \left(\begin{smallmatrix} 1 & 2 \\ 0 & 1 \end{smallmatrix}\right), 
 \left(\begin{smallmatrix} 0 & -1 \\ 1 & 0 \end{smallmatrix}\right)\right\rangle$, so $\left(\psi|\left[\left(\begin{smallmatrix} a & b \\ c & d \end{smallmatrix}\right)\right]\right)(\tau, z) = \psi(\tau,z)$. Note also that the congruence conditions imposed on $a, b, c, d$ ensure that
\[ \left(\alpha(n, \ell)(a-1)+\tfrac{c}{2},\alpha(n, \ell)b+\tfrac{d-1}{2}\right) \in \mathbb{Z}^2.\]
 Now for each $\ell \in S$, using the group law in \eqref{gplaw}, we have
 \begin{align*}
\left( \psi_{\alpha(n, \ell)}|\left[\left(\begin{smallmatrix} a&b\\c&d\end{smallmatrix}\right)\right]\right)(\tau, z) &= \left(\psi|\left[I, \left(\alpha(n, \ell), \tfrac{1}{2}\right),1\right]\left[\left(\begin{smallmatrix} a&b\\c&d\end{smallmatrix}\right),(0,0),1\right]\right)(\tau,z) \\
 &\!\!\!\!\!\!\!\!\!\!\!\!\!\!\!\!\!\!\!\!=\left(\psi|\left[\left(\begin{smallmatrix} a&b\\c&d\end{smallmatrix}\right),\left(\alpha(n, \ell)a+\tfrac{c}{2},\alpha(n, \ell)b+\tfrac{d}{2}\right),1\right]\right)(\tau,z) \\
  &\!\!\!\!\!\!\!\!\!\!\!\!\!\!\!\!\!\!\!\!=\left(\psi|\left[\left(\begin{smallmatrix} a&b\\c&d\end{smallmatrix}\right)\right]\left[\left(\alpha(n, \ell)(a-1)+\tfrac{c}{2},\alpha(n, \ell)b+\tfrac{d-1}{2}\right)\right]\left[\left(\alpha(n, \ell), \tfrac{1}{2}\right)\right]\right)(\tau,z) \\
 &\!\!\!\!\!\!\!\!\!\!\!\!\!\!\!\!\!\!\!\!= \psi_{\alpha(n, \ell)}(\tau, z).
 \end{align*}
 Hence, their product is also invariant under $\left(\begin{smallmatrix} a & b \\ c & d \end{smallmatrix}\right)$. The second claim now follows from Lemma \ref{2.2}
\end{proof}

We show that $\Phi_S^{(2w)}(\tau, z)$ is modular with respect to two particular other matrices.
\begin{Lemma} \label{2}
We have
$\left(\phi_S |\left[\left(\begin{smallmatrix} 1 & 0 \\ 1 & 1 \end{smallmatrix}\right)\right]\right) (\tau, z)= \phi_S(\tau, z)$ and
$\left(\phi_S |\left[\left(\begin{smallmatrix} 1 & n \\ 0 & 1 \end{smallmatrix}\right)\right]\right) (\tau, z)= \phi_S(\tau, z)$.
In particular, for any positive integer $w$, we have
\[\left(\left.\Phi_S^{(2w)}\right|\left(\begin{smallmatrix} 1 & 0 \\ 1 & 1 \end{smallmatrix}\right)\right)(\tau) =\Phi_S^{(2w)}(\tau) \qquad \text{and} \qquad
\left(\left.\Phi_S^{(2w)}\right|\left(\begin{smallmatrix} 1 & n \\ 0 & 1 \end{smallmatrix}\right)\right)(\tau) =\Phi_S^{(2w)}(\tau).\]
\end{Lemma}

\begin{proof}
First note that
$\left(\begin{smallmatrix} 1 & 0 \\ 1 & 1 \end{smallmatrix}\right) = 
\left(\begin{smallmatrix} 0 & 1 \\ -1 & 0 \end{smallmatrix}\right) 
\left(\begin{smallmatrix} 1 & -1 \\ 0 & 1 \end{smallmatrix}\right) 
\left(\begin{smallmatrix} 0 & -1 \\ 1 & 0 \end{smallmatrix}\right). $
Now we use the group law as in the previous lemma, together with \eqref{tr}, to see that for each $\ell \in S$,
\begin{align*}
\left(\psi_{\alpha(n, \ell)} |\left[\left(\begin{smallmatrix} 1 & 0 \\ 1 & 1 \end{smallmatrix}\right)\right]\right) (\tau, z)
&= \left(\psi|\left[\left(\begin{smallmatrix} 1 & 0 \\ 1 & 1 \end{smallmatrix}\right)\right]\left[\left(\tfrac{1}{2},0\right)\right]\left[\left(\alpha(n, \ell), \tfrac{1}{2}\right)\right] \right)(\tau,z) \\
&=\left(\psi|\left[\left(\begin{smallmatrix} 1 & -1 \\ 0 & 1 \end{smallmatrix}\right)\right]
\left[\left(\begin{smallmatrix} 0 & -1 \\ 1 & 0 \end{smallmatrix}\right)\right]
\left[\left(\tfrac{1}{2},0\right)\right]\left[\left(\alpha(n, \ell), \tfrac{1}{2}\right)\right] \right)(\tau,z) \\
&=\left(\psi|\left[\left(0, -\tfrac{1}{2}\right)\right]
\left[\left(\begin{smallmatrix} 0 & -1 \\ 1 & 0 \end{smallmatrix}\right)\right]
\left[\left(\tfrac{1}{2},0\right)\right]\left[\left(\alpha(n, \ell), \tfrac{1}{2}\right)\right] \right)(\tau,z) \\
&=\left(\psi |\left[\left(\begin{smallmatrix} 0 & -1 \\ 1 & 0 \end{smallmatrix}\right)\right]
\left[\left(-\tfrac{1}{2},0\right)\right]
\left[\left(\tfrac{1}{2},0\right)\right]\left[\left(\alpha(n, \ell), \tfrac{1}{2}\right)\right] \right)(\tau,z) \\
&=\psi_{\alpha(n, \ell)}(\tau, z).
\intertext{
Similarly, we have}
\left(\psi_{\alpha(n, \ell)} |\left[\left(\begin{smallmatrix} 1 & n \\ 0 & 1 \end{smallmatrix}\right)\right]\right) (\tau, z)
&= \left(\psi|\left[\left(\begin{smallmatrix} 1 & n \\ 0 & 1 \end{smallmatrix}\right)\right]\left[\left(0,\alpha(n, \ell)n\right)\right]\left[\left(\alpha(n, \ell), \tfrac{1}{2}\right) \right]\right)(\tau,z) \\
&= \left(\psi|\left[\left(0,\tfrac{n}{2}\right)\right] \left[\left(0,\alpha(n, \ell)n\right)\right]\left[\left(\alpha(n, \ell), \tfrac{1}{2}\right)\right] \right)(\tau,z) \\
&=\left(\psi|\left[\left(0,\left(\alpha(n, \ell)+\tfrac{1}{2}\right)n\right)\right]\left[\left(\alpha(n, \ell), \tfrac{1}{2}\right)\right] \right)(\tau,z) \\
&=\psi_{\alpha(n, \ell)}(\tau, z).
\end{align*}
In either case, the product is also fixed by the specified matrix.
\end{proof}

We can now prove the desired modularity property for $\Phi_S^{(2w)}(n\tau)$.

\begin{Lemma} \label{tada}
The nearly-holomorphic modular form $\Phi_S^{(2w)}(n\tau)$ is modular for $\Gamma_1(n)$.
\end{Lemma}
\begin{proof}
When we apply the weight-$2w$ slash operator to $\Phi_S^{(2w)}(n\tau)$, we obtain
\[(c\tau+d)^{-2w}\Phi_S^{(2w)}\left(n\cdot \tfrac{a\tau+b}{c\tau+d}\right)
=\left(\left.\Phi_S^{(2w)}\right|\left(\begin{smallmatrix} a & bn \\ c/n & d \end{smallmatrix}\right)\right)(n\tau).
\]
Thus to show that the function $\Phi_S^{(2w)}(n\tau)$ is fixed by a matrix $\left(\begin{smallmatrix} a & b \\ c & d \end{smallmatrix}\right) \in \Gamma_1(n)$, it suffices to check that $\Phi_S^{(2w)}(\tau)$ is fixed by $\left(\begin{smallmatrix} a & nb \\ c/n & d \end{smallmatrix}\right)$.
For any $\left(\begin{smallmatrix} a & b \\ c & d \end{smallmatrix}\right) \in \Gamma(2n)$, we have that $\left(\begin{smallmatrix} a & nb \\ c/n & d \end{smallmatrix}\right)$ satisfies the conditions of Lemma \ref{3.1}, and so fixes $\Phi_S^{(2w)}(\tau)$. Furthermore, when $n$ is even, $\left(\begin{smallmatrix} a & nb \\ c/n & d \end{smallmatrix}\right)$ satisfies these conditions for any
$\left(\begin{smallmatrix} a & b \\ c & d \end{smallmatrix}\right) \in \Gamma_0(2n) \cap \Gamma_1(n)$.
 In addition, Lemma \ref{2} implies that $\Phi_S^{(2w)}(n\tau)$ is fixed by $\left(\begin{smallmatrix} 1 & 0 \\ n & 1 \end{smallmatrix}\right)$ and $\left(\begin{smallmatrix} 1 & 1 \\ 0 & 1 \end{smallmatrix}\right)$. We now consider the cases when $n$ is odd and $n$ is even separately.

When $n$ is odd, we have shown that $\Phi_S^{(2w)}(n\tau)$ is invariant under the group generated by $\Gamma(2n) \cup\left\{\left(\begin{smallmatrix} 1 & 0 \\ n & 1 \end{smallmatrix}\right), \left(\begin{smallmatrix} 1 & 1 \\ 0 & 1 \end{smallmatrix}\right)\right\}$. This group is contained in $\Gamma_1(n)$ and properly contains $\Gamma_1(2n)$. For odd $n$, the index of $\Gamma_1(2n)$ in $\Gamma_1(n)$ is $3$, as can readily be computed from the formula in Proposition 1.7 of \cite{O}.
Therefore, the group in question must be all of $\Gamma_1(n)$.

When $n$ is even, we have shown that $\Phi_S^{(2w)}(n\tau)$ is invariant under the group generated by
$\Gamma_0(2n) \cap \Gamma_1(n)$ together with $\left(\begin{smallmatrix} 1 & 0 \\ n & 1 \end{smallmatrix}\right)$. In this case, the index of $\Gamma_1(2n)$ in $\Gamma_1(n)$ is $4$, so since $\Gamma_0(2n) \cap \Gamma_1(n)$ properly contains $\Gamma_1(2n)$ 
%(this intersection contains the matrix $\left(\begin{smallmatrix} n+1 & n+2 \\ n & n+1\end{smallmatrix}\right) \notin \Gamma_1(2n))$ 
it has index at most $2$ in $\Gamma_1(n)$. Thus, the group it generates with $\left(\begin{smallmatrix} 1 & 0 \\ n & 1 \end{smallmatrix}\right) \notin \Gamma_0(2n) \cap \Gamma_1(n)$ is all of $\Gamma_1(n)$.
\end{proof}

Putting this together, we prove Rose's conjecture.

\begin{proof}[Proof of Theorem \ref{con}]
The quasi-modular form
\[\left.\left(\frac{\partial}{\partial z}\right)^{2w}\phi_S(n\tau,z)\right|_{z=0}\]
appearing in Rose's theorem
is the constant term with respect to $\frac{1}{y}$ of the nearly-holomorphic modular form $\Phi_S^{(2w)}(n\tau)$. Thus, by Lemma \ref{tada} these forms are quasi-modular for $\Gamma_1(n)$.

Finally, when $n \in S$, we also need to know that the term involving the Dedekind eta-function is quasi-modular for $\Gamma_1(n)$. It is well-known that derivatives of modular forms are quasi-modular for the same group (see Section 5 of \cite{123}). We can view the Dedekind eta-function as a modular form of weight $\frac{1}{2}$ on $\mathrm{SL}_2(\mathbb{Z})$ with a multiplier system. Once we take the quotient, the additional multiplier cancels and we find that with the change of variable $\tau \mapsto n\tau$, this term is quasi-modular for $\Gamma_1(n)$.

Thus, $A_{S,n,k}(q)$ is a mixed-weight quasi-modular form for $\Gamma_1(n)$.
\end{proof}

\end{document}